\documentclass{article}  %
 \usepackage{amsmath}    
 \usepackage{amsthm}     
 \usepackage{amscd}      
 \usepackage{amssymb}    
 \usepackage{eucal}      
 \usepackage{latexsym}   
 \usepackage{graphicx}   
 \usepackage{verbatim}   

\newtheorem{theorem}{Theorem}[section]

\newtheorem{lemma}[theorem]{Lemma}
\newtheorem{corollary}[theorem]{Corollary}
\newtheorem{definition}[theorem]{Definition}

\begin{document}

\title{{The Vortex-Wave equation with a single vortex as the limit of the Euler equation.}}

\author{Clayton Bjorland\footnote{Department of Mathematics, University of Texas.  1 University Station C1200, Austin, TX, 78712-0257.  bjorland@math.utexas.edu}}

\maketitle
\begin{abstract}
In this article we consider the physical justification of the Vortex-Wave equation introduced by Marchioro and Pulvirenti, \cite{MR1098512}, in the case of a single point vortex moving in an ambient vorticity.  We consider a sequence of solutions for the Euler equation in the plane corresponding to initial data consisting of an ambient vorticity in $L^1\cap L^\infty$ and a sequence of concentrated blobs which approach the Dirac distribution.  We introduce a notion of a weak solution of the Vortex-Wave equation in terms of velocity (or primitive variables) and then show, for a subsequence of the blobs, the solutions of the Euler equation converge in velocity to a weak solution of the Vortex-Wave equation.
\end{abstract}
\section{Introduction}
The classical evolution of an incompressible, inviscid fluid is governed by the Euler equations.
In $\mathbb{R}^2$ the vorticity formulation of the Euler equations is:
\begin{align}
\partial_tw+v\cdot\nabla w=&0,\label{eeq}\\
v=K\ast w,\ \ \ w(\cdot,0)&=w_0.\notag
\end{align}
Here, $w:\mathbb{R}^2\times \mathbb{R}^+\rightarrow\mathbb{R}$ represents the vorticity of the fluid and $v:\mathbb{R}^2\times \mathbb{R}^+\rightarrow\mathbb{R}$ the velocity.  They are related through the curl operator $w=\nabla\times v$ and the Biot-Savart Law which inverts the curl operator.  In notation we keep throughout, $K$ is the Biot-Savart kernal in two dimensions given by
\begin{align}
K(x)=\frac{1}{2\pi}\frac{x^{\perp}}{|x|^2}.\notag
\end{align}
In this formulation the vorticity is transported by the velocity $v$.  The velocity is divergence free reflecting the incompressibility of the fluid.  A classic existence theorem (originally due to Yudovich \cite{Y} in bounded domains) states:

\begin{theorem}\label{classicaleuler}
Given initial vorticity $w_0\in L^1\cap L^\infty(\mathbb{R}^2)$ there exists a unique weak solution $w\in L^\infty([0,T];L^1\cap L^\infty(\mathbb{R}^2))$ for the vorticity formulation of the Euler equations (\ref{eeq}).
\end{theorem}

Many proofs of this theorem can be found in the literature, for example see \cite[Ch. 8]{MR1867882}.  In fact the literature relating to the Euler equation is quite large and instead of listing numerous results we point the interested reader to \cite{MR1867882} and \cite{MR1245492} and the references therein. 

In some situations if the vorticity is reasonable except for $N$ concentrated regions where it is very large it may be better to consider the Vortex-Wave system given by:
\begin{align}
\partial_t\omega+(u+\sum_{i=1}^N h_i)\cdot\nabla \omega=0,\ \ \ &\label{wveq}\\
u=K\ast\omega,\ \ \ \omega_0(\cdot,0)=\omega_0,\ \ \ \ &\notag\\
h_i=a_iK(x-z_i(t)),\ \ \ z_i(0)=z_0,&\notag\\
\frac{d}{dt}z_i(t)=u(z_i,t)+\sum_{i\neq j}h_j(z_i,t).\ &\notag
\end{align}
Here, $h_i$ represents the velocity of a point vortex with strength $a_i$ moving along the path $z_i$.  Each point vortex is moved by the ambient velocity associated with $\omega$ as well as the velocity associated with the other point vorticies.  In turn, the point vorticies influence the flow of the ambient vorticity.  A solution of this system consists of the ambient vorticity $\omega$ along with the paths of point vorticies given by $z_i(t)$.  

The Vortex-Wave equations were introduced by Marchioro and Pulvirenti \cite{MR1098512} where an existence theorem was given by constructing Lagrangian paths.  Concerning this system the authors of \cite{MR1098512} point out two outstanding questions:
\begin{itemize}
\item[I.] Is the solution of the Vortex-Wave equations unique?
\item[II.] Consider solutions of the Euler equations starting with a sequence of initial data which is the sum of an ambient vorticity $\omega_0$ and concentrated vorticity ``blobs.''  Do these solutions converge to a solution of the Vortex-Wave equations as the blobs converges to Dirac distributions?
\end{itemize}
Herein we consider the second question with $N=1$ but first point the interested reader to \cite{MR2529959} where the first question is studied under the assumption that the initial ambient vorticity $\omega_0$ is constant near the point vorticies.

The second question relates to the physical justification of the model as it is intended to model a sharply concentrated vorticity which is not necessarily a point vortex.  If the initial ambient vorticity $\omega_0$ is assumed zero, a satisfactory answer for the second question has been given in \cite{MR937359} and \cite{MR727203} (see also the references therein).  When $\omega_0\neq 0$ ``mixing'' of the ambient vorticity due to the concentrated vorticity complicates the issue significantly.  We consider the case of one vortex blob and ambient vorticity initially in $L^1\cap L^\infty$ then take the limit as the blob tends towards the Dirac distribution.

To start making this idea precise, consider the Euler equation with initial data $w^\epsilon_0=\omega_0+\delta_0^\epsilon$.  We assume $\omega_0\in L^\infty\cap L^1(\mathbb{R}^2)$ and $\delta_0^\epsilon=\epsilon^{-2}\chi_{\Lambda^\epsilon}(x)$ where $\Lambda^\epsilon$ is an open region such that $|\Lambda^\epsilon|=\epsilon^2$.  Here $\omega_0$ generates the ambient velocity field and $\delta_0^\epsilon$ is the ``blob'' which approximates a point vortex.  Such initial data belongs to $L^\infty\cap L^1(\mathbb{R}^2)$ (although not uniformly in $\epsilon$) so Theorem \ref{classicaleuler} implies the existence of a solution $w^\epsilon(x,t)\in L^1\cap L^\infty(\mathbb{R}^2)$ for all $t\geq0$.  The question we attack: Do these solutions tend to a solution of the Vortex-Wave system as $\epsilon\rightarrow 0$?

To answer this question we decompose the solution $w^\epsilon$ into a part which corresponds to $\delta^\epsilon_0$ and another part which corresponds to the ambient vorticity $\omega_0$.  This is done by considering, a posteriori, a linear transport equation which transports the vorticity along paths with velocity $v^\epsilon=K\ast w^\epsilon$.  This decomposition is the main topic of Subsection \ref{decomposition} and yields $\omega^\epsilon(x,t)$ and $\delta^\epsilon(x,t)$ which correspond to $\omega_0$ and $\delta^\epsilon_0$ respectively.  From here the main convergence theorem of this paper is argued in two parts: First, the solutions of the Euler equations are shown to be approximations of the Vortex-Wave equation in a precise sense, this is recorded in Lemma \ref{weakapproxlemma} below.  Second, these approximations are shown to converge to a weak solution of the Vortex-Wave equation.  This is Theorem \ref{maintheorem} below.

One aspect of the analysis we would like to point out is that the convergence arguments are made for the velocity of the solutions (sometimes called primitive variables).  In contrast equations (\ref{eeq}) and (\ref{wveq}) are written as transport equations for the vorticity.  One natural approach with transport equations is to consider the paths along which the vorticity moves, finding uniform bounds and applying compactness theorems to the paths.  In this situation it is difficult to obtain bounds on such paths because the velocity associated with the concentrated blob becomes unbounded in the limit.  To circumvent this difficulty we return to primitive variables and our convergence arguments follow the program of study initiated by DiPerna and Majda in in \cite{MR882068}, \cite{MR877643} and \cite{MR924702}.  .  

We have ``$L^1$ and $L^\infty$ control'' over the ambient vorticity $\omega^\epsilon$ but we do not have this control over the whole system because the vorticity associated with the concentrated blob approximates a Dirac distribution and therefore becomes unbounded in $L^p$ for $p>1$.  On the other hand, generalizing the arguments in \cite{MR727203} we are able to deduce considerable information about the sequence of vortex blobs by analyzing the moments:
\begin{align}
M_\epsilon(t)&=\int_{\mathbb{R}^2}x\delta^\epsilon(x,t)\,dx,\label{moment1}\\
I_\epsilon(t)&=\int_{\mathbb{R}^2}(x-M_\epsilon(t))^2\delta^\epsilon(x,t)\,dx.\label{moment2}
\end{align}
In some sense the extra information about $M_\epsilon$ and $I_\epsilon$ allow us to treat the system as having ``$L^1$ and $L^p$ control'' for $p\in (1,\infty)$.

The uniform bounds on $M_\epsilon$ and $I_\epsilon$ as $\epsilon\rightarrow 0$ are the key estimates used to show the approximate solutions converge to weak solutions of the vortex-wave equation.   In \cite{MR727203} the authors assume the blobs are moving in an ambient velocity field which is Lipschitz but that is too strong for  velocity fields generated by vorticity in $L^\infty\cap L^1$.  In the later case we only have a ``log-Lipschitz'' sense of continuity (see Definition \ref{phi} below).  Nevertheless the analysis in \cite{MR727203} is quite robust and the general ideas can be pushed through in this more general setting, this is detailed in Section \ref{momentsection}.

Before we state the main theorem of this paper we need two more assumptions concerning the initial structure of the vortex blobs.  These assumptions are not very restrictive and merely force the blob to initially behave like a reasonable approximation to the Dirac distribution.

\begin{flushleft}
{\bf Assumption:}\\
\begin{align}
\delta_0^\epsilon=\frac{1}{\epsilon^2}\chi_\Lambda^\epsilon\label{diracshape}
\end{align}
Here $\chi$ is the indicator function and the set $\Lambda^\epsilon$ has measure $\epsilon^2$.
\end{flushleft}

\begin{flushleft}
{\bf Assumption:}\\
For any continuous bounded function $f:\mathbb{R}^2\rightarrow\mathbb{R}$,
\begin{align}
\lim_{\epsilon\rightarrow 0}\int_{\mathbb{R}^2}f(y)\delta^\epsilon_0(y)\,dy = f(z_0).\label{diracIC}
\end{align}
\end{flushleft}

This assumption indicates that $\delta_0^\epsilon$ approximates the Dirac distribution at the point $z_0$ and will be assumed throughout.  We reserve the label $z_0$ specifically for this point.  The next assumption will only be used at the very end of our analysis when proving the convergence associated with the term $h^\epsilon\cdot\nabla h^\epsilon$ in Section \ref{existence}.
\begin{flushleft}
{\bf Assumption:}\\
\end{flushleft}
\begin{align}
I_\epsilon(0)\leq C\epsilon^2.\label{diracstructre}
\end{align}

We now state a lemma describing properties of the approximate sequence.  These results follow quickly from the fact that our approximate sequence is derived from solutions of the Euler equation.  The solutions stated in the lemma are exactly those sketched earlier in the introduction but made precise in Subsection \ref{decomposition}.
\begin{lemma}\label{weakapproxlemma}
Let $w^\epsilon$ be the solution of (\ref{eeq}) given by Theorem \ref{classicaleuler} with initial data $w_0=\omega_0+\delta_0^\epsilon$ where $\omega_0\in L^\infty([0,T],L^1\cap L^\infty(\mathbb{R}^2))$ and $\delta_0^\epsilon$ satisfies (\ref{diracshape}-\ref{diracstructre}).  Write $v^\epsilon=K\ast w^\epsilon$ and let $\omega^\epsilon$ and $\delta^\epsilon$ solve, respectively:
\begin{align}
\partial_t\omega^\epsilon+v^\epsilon\cdot\nabla \omega^\epsilon&=0, &\partial_t\delta^\epsilon+v^\epsilon\cdot\nabla \delta^\epsilon&=0,\notag\\
\omega^\epsilon(\cdot,0)&=\omega_0.& \delta^\epsilon(\cdot,0)&=\delta^\epsilon_0.\notag
\end{align}
Denote by $u^\epsilon=K\ast \omega^\epsilon$ and $h^\epsilon=K\ast\delta^\epsilon$ the corresponding velocities.   

Then the following properties hold:
\begin{itemize}
\item[(i)] The sequence $u^\epsilon$ is uniformly bounded in $L^\infty(\mathbb{R}^2)$, that is
\begin{align}
\sup_{t\in [0,T]}\|u^\epsilon(t)\|_{L^\infty(\mathbb{R}^2)} < C.\notag
\end{align}
\item[(ii)] $\nabla\cdot u^\epsilon =0$ and $\nabla \cdot h^\epsilon=0$ where the divergence is understood in the weak sense.
\item[(iii)] Given any test function $\Phi\in C_\sigma^\infty (\mathbb{R}^2\times [0,T])$:
\begin{align}
0=\int_0^T\int_{\mathbb{R}^2}&\Phi_t u^\epsilon+\nabla \Phi:u^\epsilon\otimes u^\epsilon +\nabla \Phi:h^\epsilon\otimes u^\epsilon\,dx\,dt,\notag\\
&+\int_0^T\int_{\mathbb{R}^2}\Phi_t h^\epsilon+\nabla \Phi:u^\epsilon\otimes h^\epsilon +\nabla \Phi:h^\epsilon\otimes h^\epsilon\,dx\,dt.\label{epsweak}
\end{align}
\item[(iv)] The sequence $\omega^\epsilon$ is uniformly bounded in $L^1\cap L^\infty(\mathbb{R}^2)$, that is
\begin{align}
\sup_{t\in [0,T]}\|\omega^\epsilon(t)\|_{L^\infty\cap L^1(\mathbb{R}^2)} < C.\notag
\end{align}
\item[(v)] The sequence $v^\epsilon$ is uniformly bounded in $Lip([0,T],H^{-L}_{loc}(\mathbb{R}^2))$ for some $L>0$, that is
\begin{align}
\|\rho u^\epsilon(t_1)-\rho u^\epsilon(t_2)\|_{H^{-L}(\mathbb{R}^2)}&\leq C|t_1-t_2|,\notag\\
0\leq t_1,t_2\leq T,\ \ \ \ \ \ \  \forall &\rho\in C^\infty_c(\mathbb{R}^2).\notag
\end{align}
\end{itemize}
\end{lemma}
In the above lemma we have used the notation $A:w\otimes v=\sum_{i,j}A_{i,j}w_iv_j$ and $(\nabla \Phi)_{ij}=\partial_i\Phi_j$.  $C^\infty_\sigma(\Omega)$ denotes the space of smooth divergence free vector fields with compact support in $\Omega$, we will also use $C^\infty_c(\Omega)$ to denote the set of smooth scalar valued functions with compact support in $\Omega$.  Condition (v) is a technical consideration introduced in \cite{MR882068}, it describes how the solution attains the initial data and also plays a role in compactness arguments.  Before stating the main convergence theorem we give the definition of a weak solution for the Vortex-Wave equation in primitive variables.  

\begin{definition}\label{weakvortexdefn}
A weak solution in primitive variables to the Vortex-Wave equation on $[0,T]$ with a single vortex consists of a path $z(t):\mathbb{R}^+\rightarrow\mathbb{R}^2$ and a velocity $u\in L^\infty([0,T],L^2_{loc}(\mathbb{R}^2))$ such that the following hold:
\begin{itemize}
\item[(i)] $\nabla \cdot u=0$ in the weak sense.
\item[(ii)]Given any test function $\Phi\in C_\sigma^\infty (\mathbb{R}^2\times [0,T]\setminus \{(z(t), t)\}_{t\in [0,T]})$,
\begin{align}
0&=\int_0^T\int_{\mathbb{R}^2}\Phi_t u+\nabla \Phi:u\otimes u+\nabla \Phi:h\otimes u \,dx\,dt\notag\\
&\ \ \ \ \ \ \ \ \ \ +\int_0^T\int_{\mathbb{R}^2}\Phi_t h+\nabla \Phi:u\otimes h \,dx\,dt\label{weak}\\
h(x,t)&=K(x-z(t))\notag
\end{align}
\item[(iii)] The velocity $u$ belongs to $Lip([0,T],H^{-L}_{loc}(\mathbb{R}^2))$ for some $L>0$ and $u(0)=u_0\in H^{-L}_{loc}$.
\end{itemize}
\end{definition}
The support of the test functions in the above definition are not allowed to contain the path $(z(t),t)$, this allows us to pass the limits through the approximate sequence avoiding some of the complications involved with a point concentration of vorticity.  $u\in L^\infty([0,T],L^2_{loc}(\mathbb{R}^2))$ is understood to mean
\begin{align}
\max_{0\leq t\leq T}\int_{B_R}|u|^2\,dx\leq C_{R,T}\notag
\end{align}
for any $R>0$.  Here (and throughout) $B_R$ denotes the ball of Radius $R$.  

The first integral in (\ref{weak}) corresponds to the first line in (\ref{wveq}) written in primitive variables while the second corresponds to the ode governing the flow of the vortex.  To interpret the solution of (\ref{wveq}) constructed through Lagrangian in this setting one should add and subtract the cross term indicated by (\ref{ueps}) and (\ref{heps}) to seperate these two parts.

Now we state the main theorem of the paper whose proof is given in Section \ref{existence}.
\begin{theorem}\label{maintheorem}
Let $u^\epsilon(x,t)$ be the sequence of solutions given by Lemma \ref{weakapproxlemma}.  Let the approximate point vortex paths be given by
\begin{align}
z^\epsilon(t)=z_0+\int_0^t u^\epsilon(z^\epsilon(s),s)\,ds.\label{zdefn}
\end{align}
Then there exists a velocity $u \in L^\infty([0,T],L^2_{loc}(\mathbb{R}^2))$, a path $z(t)\in C([0,T];\mathbb{R}^2)$, and a subsequence $\epsilon\rightarrow 0$ (which we do not relabel) such that for any $R>0$,
\begin{align}
\int_0^T\int_{B_R} |u^\epsilon -u|\,dx\,dt \rightarrow 0,\notag\\
\int_0^T\int_{B_R} |u^\epsilon -u|^2\,dx\,dt \rightarrow 0,\notag\\
\sup_{t\in [0,T]}|z^\epsilon(t) -z(t)|\rightarrow 0.\notag
\end{align}
Moreover, $u$ and $z$ are a weak solution of the Vortex-Wave equation on $[0,T]$ in the sense of Definition \ref{weakvortexdefn}.
\end{theorem}

Two final remarks.  The convergence proved herein is global in the following sense.  One fixes the interval $[0,T]$, $T>0$ is arbitrary, then following the arguments below the convergence will hold on $[0,T]$.  Concerning the terms involving $h^\epsilon\rightharpoonup h$ we point out that the convergence is given in primitive variables but our control over the blobs is in terms of the vorticity.  To bridge this gap we use a technique which is based on changing the order of integration to ``give'' the kernel $K$ to the other terms in the limit.  Heuristically, to show
\begin{align}
\int_{\mathbb{R}} \phi(x)u^\epsilon(x) h^\epsilon(x)\,dx-\int_{\mathbb{R}} \phi(x)u^\epsilon(x) h(x)\,dx\rightarrow 0\notag
\end{align}
we would consider instead
\begin{align}
\int_{\mathbb{R}}\left(\int_{\mathbb{R}}\phi(x)u^\epsilon(x)K(y-x)\,dx\right)\delta^\epsilon(y)\,dy-\int_{\mathbb{R}}\phi(x)u^\epsilon(x)K(z(t)-x)\,dx.\notag
\end{align}
As indicated by Definition \ref{weakvortexdefn} the test functions $\phi$ will be supported away from $z(t)$ so that $K(z(t)-x)$ is smooth in the domain of integration for the second integral.  The kernel $K$ has a smoothing effect on the the term $\phi u^\epsilon$ recorded in Lemma \ref{equlemma}.  For reasonable functions this smoothing effect combined with the moment bounds are enough to imply convergence, this general statement is recorded in Lemma \ref{bigguns}.  The smoothing effect of the Biot-Savart kernel has been used before to prove existence of vortex sheets, see \cite{MR1102579} (also \cite{MR1842346}, \cite{MR1254126} and \cite{MR1326916}).  In those situations a sequence of initial data approximates a bounded measure and estimates are obtained to show the vorticity does not concentrate to a point in the limit.  In our situation the vorticity is necessarily concentrated to a point on the path of the vortex $z(t)$ but this path is not in the support of test functions allowed in Definition \ref{weakvortexdefn} so we sidestep some of the complications associated with vorticity concentrating to a point.  In trade we are required to establish the movement of the vortex through (\ref{weak}).

It is evident from the statement of Lemma \ref{weakapproxlemma} and Theorem \ref{maintheorem} our analysis relies on classical theorems for the Euler equation and the work of DiPerna and Majda.  When similar claims are found in the literature we only sketch proofs or point the reader to the literature and focus our efforts on the properties which are particular to this problem or not readily available in the literature.  Specifically, many aspects of Section \ref{solncand} are direct applications of the existing literature relating to the Euler equation so we outline the proofs.  Sections \ref{momentsection} through \ref{existence} are particular to this problem so we give more detail therein.

\section{Construction of the Approximate Sequence and Solution Candidate}\label{solncand}
\subsection{Decomposition of Solutions for the Euler Equation}\label{decomposition}

Consider the vorticity formulation of the Euler equation with initial data as described in the introduction:
\begin{align}
w^\epsilon_0&=\omega_0+\delta_0^\epsilon.\notag
\end{align}
We take $\omega_0\in L^\infty\cap L^1(\mathbb{R}^2)$ and $\delta_0^\epsilon$ to satisfy (\ref{diracshape}-\ref{diracstructre}). 
We take the vortex to have unit ``strength'' so
\begin{align}
\int_{\mathbb{R}^2}\delta_0^\epsilon(x)\,dx=1 \ \ \ \ \ \ \ \forall \epsilon>0.\notag
\end{align}
This is not a restriction as the analysis herein can be carried out replacing $\delta_0^\epsilon$ with $a\delta_0^\epsilon$ where $a\in\mathbb{R}^2$.

The initial data $w^\epsilon_0$ belongs to $L^\infty\cap L^2(\mathbb{R}^2)$ (though not uniformly in $\epsilon$) so Theorem \ref{classicaleuler} implies the existence of a solution $w^\epsilon(x,t)\in L^\infty([0,T];L^1\cap L^\infty(\mathbb{R}^2))$.  Let $v^\epsilon=K\ast w^\epsilon$ be the associated velocity.  Now consider a related linear transport equation:
\begin{align}
\partial_tf+v^\epsilon\cdot\nabla f&=0.\label{aleq}
\end{align}
Solve separately this equation with initial data $\omega_0$ and $\delta^\epsilon$.  Denote the solution of this equation corresponding to initial data $\omega_0(x)$ by $\omega^\epsilon(x,t)$ and the solution corresponding to $\delta_0^\epsilon(x)$ by $\delta^\epsilon(x,t)$.  Furthermore, let $u^\epsilon(x,t)=K\ast \omega^\epsilon$ and $h^\epsilon=K\ast\delta^\epsilon$ denote the associated velocities respectively.  Using the linearity of equation (\ref{aleq}) and the uniqueness of solutions for the Euler equation given by Theorem \ref{classicaleuler} we can conclude
\begin{align}
w^\epsilon(x,t)&=\omega^\epsilon(x,t)+\delta^\epsilon(x,t),\notag\\
v^\epsilon(x,t)&=u^\epsilon(x,t)+h^\epsilon(x,t).\notag
\end{align}
Equation (\ref{aleq}) is a transport equation so the divergence free property of $u$ implies solutions will conserve their $L^p$ norms.  We record this property here because we will use it later:
\begin{align}
\|\omega^\epsilon(t)\|_{L^p(\mathbb{R}^2)}&\leq \|\omega_0\|_{L^p(\mathbb{R}^2)}\ \ \ \forall p\in [1,\infty],\label{transportconservation}\\
\|\delta^\epsilon(t)\|_{L^1(\mathbb{R}^2)}&\leq \|\delta^\epsilon_0\|_{L^1(\mathbb{R}^2)}=1.\label{htransportconserve}
\end{align}

\subsection{Solution Candidate}

Returning to primitive variables we rewrite the solutions constructed in the previous subsection:
\begin{align}
\partial_t u^\epsilon + u^\epsilon\cdot\nabla u^\epsilon + h^\epsilon\cdot\nabla u^\epsilon-\nabla u^\epsilon\cdot h^\epsilon+\nabla p_u&=0,\label{ueps}\\
\partial_t h^\epsilon + u^\epsilon\cdot\nabla h^\epsilon + h^\epsilon\cdot\nabla h^\epsilon+\nabla u^\epsilon\cdot h^\epsilon+\nabla p_h&=0.\label{heps}
\end{align}
Here $u^\epsilon$ and $h^\epsilon$ are divergence free functions and $p_u$ and $p_h$ are the associated pressures.  The term $\nabla u^\epsilon\cdot h^\epsilon=\sum_i h^\epsilon_i\nabla u^\epsilon_i$ balances the other terms involving $u^\epsilon$ and $h^\epsilon$ so that when the curl is taken one recovers the transport equations governing $\omega^\epsilon$ and $\delta^\epsilon$.  These equations should be interpreted in a weak sense.   We now sketch the proof of Lemma \ref{weakapproxlemma}.
\begin{proof}(Proof of Lemma \ref{weakapproxlemma}).
These statements follow directly from well known properties of the Biot-Savart kernel and solutions of the Euler equation so we only indicate parts of the proof here.  Properties (ii) and (iii) are just the weak interpretation of (\ref{ueps}) and (\ref{heps}) but added together.  Property (iv) is a direct consequence of (\ref{transportconservation}).  By a well known estimate of the Biot-Savart kernel (see \cite[p. 313]{MR1867882}) Property (iv) implies Property (i).  

Checking Property (v) is straightforward using (\ref{ueps}) and following Appendix A in \cite{MR882068} (or \cite[p. 395]{MR1867882}).  The main difference in this setting are the terms with $h^\epsilon$ which are handled in the following way.  First, (\ref{htransportconserve}) implies the sequence $\delta^\epsilon$ is uniformly bounded in $L^\infty([0,T];L^1(\mathbb{R}^2))$.  Also, the Biot-Savart kernel $K\in L^p_{loc}$ for any $p\in [1,2)$ so
\begin{align}
\sup_{t\in[0,T]}\|\rho h^\epsilon(t)\|_{L^p(\mathbb{R}^2)}<C \ \ \ \ \ \ \ \forall \rho\in C^\infty_c(\mathbb{R}^2),\ \ p\in [1,2). \label{comphbound}
\end{align}
In this bound the constant $C$ depends on the support of $\rho$ but not on $\epsilon$.  Combining this with Item (i) we deduce
\begin{align}
\sup_{t\in[0,T]}\|\rho (h^\epsilon(t)\otimes u^\epsilon(t))\|_{L^p(\mathbb{R}^2)}<C \ \ \ \ \ \ \ \forall \rho\in C^\infty_c(\mathbb{R}^2),\ \ p\in [1,2).\notag
\end{align}
Writing $\nabla u^\epsilon$ in terms of $\omega^\epsilon$ using a singular integral operator, then using the Calderon-Zygmond inequality (see \cite[p. 73]{MR1867882} and \cite{MR0290095}) one finds
\begin{align}
\|\nabla u^\epsilon(t)\|_{L^p(\mathbb{R}^2)}\leq C_p\|\omega^\epsilon(t)\|_{L^p(\mathbb{R}^2)}\notag
\end{align}
for any $p\in (1,\infty)$.
Combining this with property (iv) and (\ref{comphbound}) establishes 
\begin{align}
\sup_{t\in[0,T]}\|\rho (\nabla u^\epsilon(t)\cdot h^\epsilon(t))\|_{L^p(\mathbb{R}^2)}<C \ \ \ \ \ \ \ \forall \rho\in C^\infty_c(\mathbb{R}^2),\ \ p\in (1,2).\notag
\end{align}
This is enough to establish (v) following the arguments in Appendix A of \cite{MR882068}.\qed
\end{proof}

With the properties established in Lemma \ref{weakapproxlemma} we can use Theorems 10.1 and 10.2 in \cite{MR1867882} to find the solution candidate.  This next lemma establishes many of the properties stated in Theorem \ref{maintheorem}.
\begin{lemma}\label{uexistence}
There exists a function $u\in L^\infty([0,T],L^2_{loc}(\mathbb{R}^2))$ and a subsequence $u^\epsilon$ (which we do not relabel) such that
\begin{align}
\iint_\Omega |u^\epsilon -u|\,dx\,dt \rightarrow 0,\label{l1convergence}\\
\iint_\Omega |u^\epsilon -u|^2\,dx\,dt \rightarrow 0.\label{l2convergence}
\end{align} 
\end{lemma}
\begin{proof}
This is a direct application of the work done in proving Theorems 10.1 and 10.2 in \cite{MR1867882}.  Those theorems in turn rely in Properties (i), (iv), and (v) in Lemma \ref{weakapproxlemma} in combination with the Lions-Aubin Compactness Lemma, the Calderon-Zygmund Inequality, and the Sobolev Inequality.  \qed
\end{proof}
This lemma rules out ``oscillations'' in the limiting process.  In \cite{MR1867882}, Theorems 10.1 and 10.2, the authors assume an additional comparability with the Euler equation and deduce existence of a solution for the Euler equation.  Here Property (iii) of Lemma \ref{weakapproxlemma} holds instead and we prove later that the function $u$ given above is a weak solution for the Vortex-Wave equation in the sense of Definition \ref{weakvortexdefn}

\subsection{Path of the Point Vortex}

Now we are in a position to construct the path of the point vortex.  
\begin{lemma}\label{pathlemma}
Let $u^\epsilon(x,t)$ be the sequence of solutions given by Lemma \ref{weakapproxlemma}.  Furthermore, let $z^\epsilon(t)$ be defined by (\ref{zdefn}) for $\epsilon\geq 0$.  Then, there exists a subsequence $z^\epsilon$ and a continuous function $z(t)$ such that
\begin{align}
\sup_{t\in [0,T]}|z^\epsilon(t)-z(t)|\rightarrow 0.\notag
\end{align}
\end{lemma}
\begin{proof}
Property (i) of Lemma \ref{weakapproxlemma} implies $u^\epsilon\in L^\infty(\mathbb{R}^2)$ uniformly in time and uniformly in $\epsilon$.  Therefore $z^\epsilon(t)$ is a sequence of equicontinuous and uniformly bounded functions on some set $B_R(z_0)\times [0,T]$.  The Arzella-Ascoli theorem gives uniform convergence of a subsequence to a continuous function $z(t)$.  \qed
\end{proof}

In the above lemma we do not prove $z(t)$ solves $\dot{z}(t)=u(z(t),t)$ although that is what we expect.  In the final section we will show $h(x,t)=K(z(t)-x)$ satisfies (\ref{weak}) which contains some of the information from this ODE.

\section{Estimates for the Vorticity Moments}\label{momentsection}
In this section we work in a slightly simplified situation and prove a theorem which is a generalization of Theorem 2.1 in \cite{MR727203}.  Later it will be applied to the sequence constructed in the proceeding section.  We consider a single vortex blob moving in an external field $F$.

\begin{definition}\label{phi}
We say a vector field $v(x)$ is ``log-Lipschitz'' if $|v(x)-v(y)|\leq C\phi(|x-y|)$ for all $x,y$ where $\phi$ is given by:
\begin{align}
\phi(|x|)=\left\{ 
\begin{array}{lcl}
|x|(1-\log|x|) &\text{if} &|x|<1\\
1 &\text{otherwise}
\end{array}\right .\notag
\end{align}
\end{definition}

Let $F(x,t):\mathbb{R}^2\times\mathbb{R}^+\rightarrow \mathbb{R}^2$ be a divergence free ``uniformly log-Lipschitz'' vector field, that is $|F(x,t)-F(y,t)|\leq C_F\phi(|x-y|)$ for all $t\in[0,T]$.
In the subsequent chapters we will choose $F=u^\epsilon$ to apply the analysis of this section to the general situation.  It is well known that $\omega\in L^\infty([0,T];L^1\cap L^\infty(\mathbb{R}^2))$ implies $u=K\ast \omega$ is uniformly log-Lipschitz with constant depending only on the $L^1\cap L^\infty$ norm of $\omega$, (see \cite[p. 315]{MR1867882}).  Therefore we are careful throughout to bound only using $C_F$ as described above so that when we replace $F$ by $u^\epsilon$ we will have uniform bounds.

Within this section $\delta^\epsilon(x,t)$ denotes the solution of
\begin{align}
\partial_t\delta^\epsilon(x,t)&+(h^\epsilon+F)\cdot\nabla \delta^\epsilon=0,\notag\\
\delta^\epsilon(x,0)&=\epsilon^{-2}\chi_{\Lambda^\epsilon}(x),\notag\\
h^\epsilon&=K\ast \delta^\epsilon.\notag
\end{align}
Associated with this system are the equations for the trajectory $x(t,x_0)$ of a particle initially at position $x_0$ and moved by the velocity field $h^\epsilon+F$:
\begin{align}
\frac{d}{dt}x^\epsilon(t,x_0)&=h^\epsilon(x^\epsilon(t,x_0),t)+F(x^\epsilon(t,x_0),t),\notag\\
x(0,x_0)&=x_0.\notag
\end{align}
The first system above is the situation viewed in Eulerian coordinates while the second is the situation viewed in Lagrangian coordinates.  We will use both.  For the trajectory starting at $z_0$ (corresponding to the initial position of the blob) we reserve the notation $z^\epsilon(t)$.

\subsection{Uniform Moment Bounds}

Define:
\begin{align}
M_\epsilon(t)&=\int_{\mathbb{R}^2}x\delta^\epsilon(x,t)\,dx=\epsilon^{-2}\int_{\Lambda^\epsilon}x^\epsilon(t,x)\,dx,\notag\\
I_\epsilon(t)&=\int_{\mathbb{R}^2}(x-M_\epsilon(t))^2\delta^\epsilon(x,t)\,dx=\epsilon^{-2}\int_{\Lambda^\epsilon}(x^\epsilon(t,x)-M_\epsilon(t))^2\,dx.\notag
\end{align}
The first goal of this section is to investigate how these quantities change through motion in the ambient fluid $F$.  We note that in the absence of the ambient vector field ($F=0$) both quantities are conserved.

Let $\Lambda_t^\epsilon$ be the image of $\Lambda^\epsilon$ under the mapping $x_0\rightarrow x^\epsilon(t,x_0)$.  Then,
\begin{align}
\frac{d}{dt}M_\epsilon(t)=\epsilon^{-2}\int_{\Lambda^\epsilon_t}F(x,t)\,dx.\label{Mbound}
\end{align}
The above relation is exactly equation (2.10) in \cite{MR727203} so we offer it without proof.  Also,
\begin{align}
\frac{d}{dt}I_\epsilon(t)=2\epsilon^{-2}\int_{\Lambda^\epsilon}(x^\epsilon(t,x)-M_\epsilon(t))\cdot F(x^\epsilon(t,x),t)\,dx.\notag
\end{align}
Using the log-Lipschitz property of $F$ yields
\begin{align}
|\frac{d}{dt}I_\epsilon(t)|\leq 2C_F \phi(I_\epsilon(t))\notag.
\end{align}
The above two relations are exactly (2.12) and (2.13) in \cite{MR727203} but using the log-Lipschitz property instead of the Lipschitz property.
If $I_\epsilon(0)\leq \exp(-\exp C_FT)$ this implies (see for example \cite[p. 319]{MR1867882}): 
\begin{align}
I_\epsilon(t)\leq eI_\epsilon(0)^{\exp(-C_Ft)}&& \forall t\in[0,T].\label{I2bound}
\end{align}
Assumption (\ref{diracIC}) assures this condition on $I_\epsilon(0)$ can be enforced by taking $\epsilon$ sufficiently small once $T$ is fixed.

\subsection{Convergence of First Moments}
\begin{theorem}\label{momconvthrm}
Let $\delta^\epsilon_0$ satisfy assumption (\ref{diracIC}) and fix $T>0$.  Then, 
\begin{align}
\lim_{\epsilon\rightarrow 0}\sup_{t\in [0,T]}|M_\epsilon(t)-z^\epsilon(t)|=0.\notag
\end{align}
In fact, if $I_\epsilon(0)\leq \exp(-\exp C_FT)$, then
\begin{align}
|M_\epsilon(t)-z^\epsilon(t)| \leq e\left(|z_0-M_\epsilon(0)|+pTC_F(eI_\epsilon(0)^{\exp(-C_Ft)})^{\frac{1}{2q}}\right)^{\exp(-C_Ft)}\label{momentconvest}
\end{align}
for all $t\in [0,T]$ where $\frac{1}{p}+\frac{1}{q}=1$, and $p,q\in (1,\infty)$.
\end{theorem}
The bound (\ref{momentconvest}) is recorded because it will be used to show how $h^\epsilon\otimes h^\epsilon\rightarrow h\otimes h$ in the final section.
\begin{proof}
This proof is argued essentially as Theorem 2.1 in \cite{MR727203} but using the weaker assumption on $F$.  Assumption (\ref{diracIC}) implies $M_\epsilon(0)\rightarrow z_0$ and $I_\epsilon(0)\rightarrow 0$ so we may choose $\epsilon$ small enough that $I_\epsilon(0)\leq \exp(-\exp C_FT)$ and (\ref{I2bound}) holds.  We begin with a consequence of (\ref{Mbound}):
\begin{align}
|z^\epsilon(t)-M_\epsilon(t)|\leq |z_0&-M_\epsilon(0)|+\int_0^t|F(z^\epsilon(s),s)-F(M_\epsilon(s),s)|\,ds \label{diffMbound}\\
&\ \ \ \ \ +\int_0^t|F(M_\epsilon(s),s)-\epsilon^{-2}\int_{\Lambda^\epsilon_s}F(x,s)\,dx|\,ds.\notag
\end{align}
To prove the theorem we estimate the two integrals on the right hand side and use a Gronwall argument to achieve the desired bound.

The following estimate is used to bound the second integral on the right hand side of (\ref{diffMbound}), it holds for any $p\in(1,\infty)$:
\begin{align}
\phi(t)\leq pt^{1-\frac{1}{p}}.\label{alb}
\end{align}
Now,
\begin{align}
F(M_\epsilon(s),s)=\epsilon^{-2}\int_{\Lambda^\epsilon_s}F(M_\epsilon(s),s)\,dx,\notag
\end{align}
so that the log-Lipschitz continuity along with (\ref{alb}) (let $p\in(1,\infty)$ be arbitrary then set $\frac{1}{p}+\frac{1}{q}=1$) and H\"older's inequality imply
\begin{align}
|F(M_\epsilon(s),s)-\epsilon^{-2}\int_{\Lambda^\epsilon_s}F(x,s)\,dx|&\leq \epsilon^{-2}C_F\int_{\Lambda^\epsilon_s}\phi(M_\epsilon(s)-x)\,dx\notag\\
&\leq C_F\epsilon^{-2}p\int_{\Lambda^\epsilon_s}(M_\epsilon(s)-x)^{1-\frac{1}{p}}\,dx\notag\\
&\leq C_F\epsilon^{\frac{1}{q}-2}p|\Lambda^\epsilon_s|^{\frac{1+p}{2p}}I_\epsilon(s)^{\frac{1}{2q}}.\notag
\end{align}
The external vector field $F$ is divergence free so $|\Lambda^\epsilon_s|=|\Lambda^\epsilon|=\epsilon^2$ and
\begin{align}
|F(M_\epsilon(s),s)-\epsilon^{-2}\int_{\Lambda^\epsilon_s}F(x,s)|\leq pC_FI_\epsilon(s)^{\frac{1}{2q}}.\notag
\end{align}
Combining this bound with (\ref{I2bound}) allows
\begin{align}
\int_0^t|F(M_\epsilon(s),s)-\epsilon^{-2}\int_{\Lambda^\epsilon_s}F(x,s)|\,ds \leq pTC_F(eI_\epsilon(0)^{\exp(-C_Ft)})^{\frac{1}{2q}} && \forall t\in[0,T].\notag
\end{align}
The first integral on the right hand side of (\ref{diffMbound}) is estimated using log-Lipschitz continuity:
\begin{align}
\int_0^t|F(z^\epsilon(s),s)-F(M_\epsilon(s),s)|\,ds \leq \int_0^tC_F\phi(z^\epsilon(s)-M_\epsilon(s))\,ds.\notag
\end{align}
With these two estimates in hand (\ref{diffMbound}) becomes
\begin{align}
|z^\epsilon(t)-M_\epsilon(t)|\leq |z_0&-M_\epsilon(0)|+pTC_F(eI_\epsilon(0)^{\exp(-C_Ft)})^{\frac{1}{2q}}\label{usegronwall}\\
 &\ \ \ \ \ +\int_0^tC_F\phi(z^\epsilon(s)-M_\epsilon(s))\,ds.\notag
\end{align}
To handle this we use a generalized Gronwall inequality which we prove after the conclusion of this proof.

\begin{lemma}\label{gwlemma}
Let $C_0,C_1>0$ satisfy $C_0\leq\exp(-\exp C_1T)$.  If $f:[0,T]\rightarrow \mathbb{R}^+$ satisfies the bound
\begin{align}
f(t)\leq C_0+ \int_0^tC_1\phi(f(s))\,ds,\notag
\end{align}
then $f(t)\leq eC_0^{\exp(-C_1t)}$ for all $t\in [0,T]$.
\end{lemma}

The first two terms in (\ref{usegronwall}) tend to zero as $\epsilon\rightarrow 0$, so after applying Lemma \ref{gwlemma} we have finished the proof.  We have not used any quantitative property of $F$ aside from the constant $C_F$.\qed
\end{proof}

\begin{proof}(Proof of Lemma \ref{gwlemma})
Since $\phi$ is an increasing function, if $g\geq 0$ satisfies 
\begin{align}
g(t)=C_0+ C_1\int_0^t\phi(f)\,dx\notag
\end{align}
we may conclude $f(t)\leq g(t)$.  Such a $g$ would also satisfy
\begin{align}
g'(t)\leq C_1\phi(g)\ \ \ \ \ \ \ \ \ g(0)=C_0\notag
\end{align}
which can be solved to yield $g(t)\leq eC_0^{\exp(-C_1t)}$.  See for example \cite[p319]{MR1867882}.\qed
\end{proof}

\section{Technical Limit Lemmas}\label{techsec}
To compute the limits involving $h^\epsilon$ we will make use of technical lemmas which describe how the convergence of $\delta^\epsilon$ can be used to imply the convergence of $h^\epsilon$.  
\begin{definition}\label{ueqdef}
We say a sequence of functions $\{f^\alpha(y,s)\}_{\alpha\in A}$ is ``uniformly equicontinuous, uniformly on $[0,T]$'' if given any $\epsilon >0$ there exists $\eta>0$ such that if $|y-\bar{y}|<\eta$ then $|f^{\alpha}(y,s)-f^{\alpha}(\bar{y},s)|<\epsilon$ for all $\alpha\in A$, $y,\bar{y}\in \mathbb{R}^2$ and $s\in [0,T]$.
\end{definition}

To give an idea of how such functions will arise in our analysis, in the next lemma we prove that a reasonable sequence of functions in convolution with the Biot-Savart kernel satisfy the above definition.

\begin{lemma}\label{equlemma}
Let $g^\alpha:\mathbb{R}^+\times \mathbb{R}^2\rightarrow \mathbb{R}$ satisfy  
\begin{align}
\|g^\alpha\|_{L^\infty([0,T];L^p(\mathbb{R}^2))}<C<\infty\notag
\end{align}
for some $p\in (2,\infty)$.  The constant is assumed independent of $\alpha$.  Then each component of $K\ast g^\alpha$ is uniformly equicontinuous, uniformly on $[0,T]$ in the sense of Definition \ref{ueqdef}.
\end{lemma}
\begin{proof}
Let $x^1,x^2\in \mathbb{R}^2$ satisfy $d=|x^1-x^2|<1$.  We will show $|K\ast g^\alpha(x^1,t)-K\ast g^\alpha(x^2,t)|$ is uniformly bounded by some power of $d$.  This actually proves some type of H\"older continuity but we only use the weaker condition.  The proof follows typical arguments to show log-Lipschitz continuity (see for example \cite{MR1867882} Lemma 8.1) and is included here because the exact result is not readily available in the literature.  To begin we split the integral into three parts:
\begin{align}
|K\ast g^\alpha(x^1,t)-K\ast g^\alpha(x^2,t)| &\leq \int_{\mathbb{R}^2\setminus B_2(x^1)}+\int_{B_2(x^1)\setminus B_{2d}(x^1)}+\int_{B_{2d}(x^1)}\notag\\ 
&\ \ \ \ \ \ \ \ \ \ \times |K(x^1-y)-K(x^2-y)||g^\alpha(y,t)|\,dy\notag\\
&\leq A_1+A_2+A_3.\notag
\end{align}

To bound $A_1$ we use the estimate (\cite[p. 317]{MR1867882})
\begin{align}
|K(x)-K(y)|\leq \frac{1}{\pi}\frac{|x-y|}{|x||y|}.\notag
\end{align}
H\"older's inequality with $\frac{1}{p}+\frac{1}{q}=1$ implies
\begin{align}
A_1\leq  \frac{d}{\pi} \int_{\mathbb{R}^2\setminus B_2(x^1)}\frac{g^\alpha(y)}{|x^1-y||x^2-y|}\,dy \leq\frac{d}{\pi}\|g^\alpha\|_{L^p(\mathbb{R}^2)}\left(\int_1^\infty\frac{\,dr}{r^{2q-1}}\right)^{\frac{1}{q}}.\notag
\end{align}
The integral on the right hand side is finite when $q\in (1,\infty)$ and in turn $p\in (1,\infty)$.  This implies
\begin{align}
A_1\leq cd\|g^\alpha\|_{L^p(\mathbb{R}^2)}.\label{A1bound}
\end{align}

For $y\in B_2(x^1)\setminus B_{2d}(x^1)$ the mean-value theorem implies (\cite[p. 318]{MR1867882})
\begin{align}
|K(x^1-y)-K(x^2-y)|\leq c\frac{|x^1-x^2|}{|x^1-y|^2}.\notag
\end{align}
Again using H\"older's inequality,
\begin{align}
A_2\leq cd\|g^\alpha\|_{L^p(\mathbb{R}^2)}\left(\int_{2d}^2\frac{\,dr}{r^{2q-1}}\right)^{\frac{1}{q}}.\notag
\end{align}
The integral on the right hand side is finite when $q\in (1,\infty)$ but we require $q\in (1,2)$ (and hence $p\in (2,\infty)$) to retain positive powers of $d$ in our estimate and establish continuity.  Indeed,
\begin{align}
A_2\leq c_pd^{\frac{2}{q}-1}\|g^\alpha\|_{L^p(\mathbb{R}^2)}.\label{A2bound}
\end{align}

We bound the third term using H\"older's inequality and the estimate $|K(x)|\leq C|x|^{-1}$.  So,
\begin{align}
&\left(\int_{B_{2d}(x^1)}|K(x^1-y)-K(x^2-y)|^q\right)^{\frac{1}{q}} \notag\\
&\ \ \ \ \ \ \ \ \ \ \ \leq \left(\int_{B_{2d}(x^1)}\frac{\,dy}{|x^1-y|^{q}}\right)^{\frac{1}{q}}+\left(\int_{B_{2d}(x^1)}\frac{\,dy}{|x^2-y|^{q}}\right)^{\frac{1}{q}}\notag\\
&\ \ \ \ \ \ \ \ \ \ \ \leq 2\left(\int_0^{2d}\frac{\,dr}{r^{q-1}}\right)^{\frac{1}{q}}.\notag
\end{align}
This integral is finite if $q\in (1,2)$, so if $p\in (2,\infty)$ we can bound
\begin{align}
A_3\leq c_pd^{\frac{2}{q}-1}\|g^\alpha\|_{L^p(\mathbb{R}^2)}.\label{A3bound}
\end{align}

Considering (\ref{A1bound}-\ref{A3bound}) we see $g^\alpha\ast K$ satisfies the conclusion of this lemma.\qed
\end{proof}

When we do not need to deal with a sequence of functions we have the following corollary and definition.

\begin{definition}\label{ucudef}
We say a function $f(x,t)$ is ``uniformly continuous, uniformly on $[0,T]$'' if given any $\epsilon >0$ there exists $\eta>0$ such that if $|y-\bar{y}|<\eta$ then $|f(y,s)-f(\bar{y},s)|<\epsilon$ for all $y,\bar{y}\in \mathbb{R}^2$ and $s\in [0,T]$.
\end{definition}

\begin{corollary}\label{equcor}
If $g\in L^\infty([0,T];L^p(\mathbb{R}^2))$ then $g\ast K$ is uniformly continuous, uniformly on $[0,T]$.
\end{corollary}
\begin{proof}
Choose $g^\alpha=g$ for all $\alpha$ in Lemma \ref{equlemma}.\qed
\end{proof}

Now we state the main technical device for proving convergence.
\begin{lemma}\label{bigguns}
Let $f^\alpha(y,s)$ be a sequence of functions, uniformly equicontinuous, uniformly on $[0,T]$, satisfying the bound
\begin{align}
\sup_{s\in [0,T]}\|f^\alpha(\cdot,s)\|_{L^\infty(\mathbb{R}^2)}\leq C\label{eqlass4}
\end{align}
Let $\delta^\alpha(x,s)$ be an approximation to the Dirac distribution which satisfies the following properties:
\begin{align}
\lim_{\alpha\rightarrow 0}\sup_{s\in [0,T]}\left|M_\alpha(s)-z(s)\right|&=\lim_{\alpha\rightarrow 0}\sup_{s\in [0,T]}\left|\int_{\mathbb{R}^2}y\delta^\alpha(y,s)\,dy-z(s)\right| = 0\label{eqlass1}\\
\lim_{\alpha\rightarrow 0}\sup_{s\in [0,T]}I_\alpha(t)=\lim_{\alpha\rightarrow 0}&\sup_{s\in [0,T]}\left(\int_{\mathbb{R}^2}(y-M_\alpha(s))^2\delta^\alpha(y,s)\,dx\right)=0\label{eqlass2}\\
\int_{\mathbb{R}^2}\delta^\alpha(x,s)\,dx &=1 \ \ \ \ and \ \ \ \ \delta^\alpha(x,s)\geq0.\label{eqlass3}
\end{align}
Then,
\begin{align}
\lim_{\alpha\rightarrow 0}\sup_{s\in [0,T]}\left|f^\alpha(z(s),s)-\int_{\mathbb{R}^2}f^\alpha(y,s)\delta^\alpha(y,s)\,dy\right|=0.\notag
\end{align}
\end{lemma}
\begin{proof}
Let $\epsilon>0$ be fixed and let $\eta>0$ be given by Definition (\ref{ueqdef}).  Using (\ref{eqlass1}) we may choose $\alpha_1>0$ small enough so that if $\alpha<\alpha_1$,
\begin{align}
\sup_{s\in [0,T]} |M_\alpha(s)-z(s)| <\eta.\notag
\end{align}
Now,
\begin{align}
|f^\alpha&(z(s),s)-\int_{\mathbb{R}^2}f^\alpha(y,s)\delta^\alpha(y,s)\,dy|\label{eps1}\\
&\leq |f^\alpha(z(x),s) -f^\alpha(M_\alpha(s),s)| +|f^\alpha(M_\alpha(s),s) - \int_{\mathbb{R}^2}f^\alpha(y,s)\delta^\alpha(y,s)\,dy|.\notag
\end{align}
If $\alpha<\alpha_1$, the first term on the right hand side is bounded uniformly in $s$ by $\epsilon$.  To attack the second term, first note that (\ref{eqlass3}) implies
\begin{align}
f^\alpha(M_\alpha(s),s) &- \int_{\mathbb{R}^2}f^\alpha(y,s)\delta^\alpha(y,s)\,dy\notag\\
&=\int_{\mathbb{R}^2}[f^\alpha(y,s)-f^\alpha(M_\alpha(s),s)]\delta^\alpha(y,s)\,dy.\notag
\end{align}
We may then decompose
\begin{align}
\int_{\mathbb{R}^2}[f^\alpha(y,s)-&f^\alpha(M_\alpha(s),s)]\delta^\alpha(y,s)\,dy\notag\\
&= \int_{B_\eta(M_\alpha(s))}[f^\alpha(y,s)-f^\alpha(M_\alpha(s),s)]\delta^\alpha(y,s)\,dy\notag\\
&\ \ \ \ \ \ +\int_{B^C_\eta(M_\alpha(s))}[f^\alpha(y,s)-f^\alpha(M_\alpha(s),s)]\delta^\alpha(y,s)\,dy.\notag
\end{align}
Relying on the uniform equicontinuity and using again (\ref{eqlass3}) we see
\begin{align}
|\int_{B_\eta(M_\alpha(s))}[f^\alpha(y,s)-f^\alpha&(M_\alpha(s),s)]\delta^\alpha(y,s)\,dy|\notag\\
&\leq \int_{B_\eta(M_\alpha(s))}\epsilon\delta^\alpha(y,s)\,dy \leq \epsilon.\label{eps2}
\end{align}
This bound is independent of $s\in [0,T]$.  To handle the other term note
\begin{align}
|\int_{B^C_\eta(M_\alpha(s))}[f^\alpha(y,s)-f^\alpha&(M_\alpha(s),s)]\delta^\alpha(y,s)\,dy|\notag\\
&\leq \frac{C}{\eta^2}\int_{\mathbb{R}^2}|y-M_\alpha(s)|^2\delta^\alpha(y,s)\,dy\notag\\
&\leq\frac{C}{\eta^2}I_\alpha(s).\notag
\end{align}
Above, $C$ is again as in (\ref{eqlass4}).  In view of (\ref{eqlass2}), there exists an $\alpha_\epsilon\leq \alpha_1$ such that for all $\alpha\leq \alpha_\epsilon$, $I_\alpha(s)\leq \frac{\eta^2\epsilon}{C}$, and
\begin{align}
|\int_{B^C_\eta(M_\alpha(s))}[f^\alpha(y,s)-f^\alpha(M_\alpha(s),s)]\delta^\alpha(y,s)\,dy|\leq \epsilon.\label{eps3}
\end{align}
This bound is independent of $s\in [0,T]$.  Combining $(\ref{eps1}-\ref{eps3})$ shows that if $\alpha<\alpha_\epsilon$,
\begin{align}
|f^\alpha(z(s),s)-\int_{\mathbb{R}^2}f^\alpha(y,s)\delta^\alpha(y,s)\,dy| <3\epsilon.\notag
\end{align}
As $\epsilon$ was picked arbitrarily, the lemma is proved.\qed
\end{proof}

\begin{corollary}\label{biggunscor}
Let $\delta^\alpha$ satisfy the assumptions of Lemma \ref{bigguns}.  If $f:\mathbb{R}^+\times\mathbb{R}^2\rightarrow \mathbb{R}$ is uniformly continuous, uniformly on $[0,T]$ and $f\in L^\infty([0,T]\times\mathbb{R}^2)$ then
\begin{align}
\lim_{\alpha\rightarrow 0}\sup_{s\in [0,T]}\left|f(z(s),s)-\int_{\mathbb{R}^2}f(y,s)\delta^\alpha(y,s)\,dy\right| =0.\notag
\end{align}
\end{corollary}
\begin{proof}
Choose $f^\alpha=f$ for all $\alpha$ in Lemma \ref{bigguns}.\qed
\end{proof}

The final point to address here is the assumption (\ref{eqlass4}) in Lemma \ref{bigguns}.  To show how this will be satisfied we present a well known estimate.
\begin{lemma}\label{inftylemma}
Let $g\in L^p\cap L^q(\mathbb{R}^2)$ where $1\leq p< 2<q\leq \infty$.  Then
\begin{align}
\|g\ast K\|_{L^\infty(\mathbb{R}^2)}\leq C(\|g\|_{L^p(\mathbb{R}^2)}+\|g\|_{L^q(\mathbb{R}^2)}).\notag
\end{align}
\end{lemma}
\begin{proof}
The case $p=1$ and $q=\infty$ is usually what is presented in the literature so we do not give a proof of those cases (see for example \cite[Prop. 8.2]{MR1867882}).  Let $B_1$ be a ball of unit radius centered at the origin.  Using the estimate $|K(x)|\leq C|x|^{-1}$ we bound, 
\begin{align}
&\|g\ast K\|_{L^\infty(\mathbb{R}^2)} \leq \int_{B_1}|g(y)K(x-y)|\,dy + \int_{B^c_1}|g(y)K(x-y)|\,dy\notag\\
&\ \ \ \ \ \leq C\|g\|_{L^q(\mathbb{R}^2)}\left( \int_{B_1}|x|^{-q^*}\,dy\right)^{\frac{1}{q^*}}+C\|g\|_{L^p(\mathbb{R}^2)}\left( \int_{B^c_1}|x|^{-p^*}\,dy\right)^{\frac{1}{p^*}}\notag
\end{align}
Here $q^*$ satisfies the usual relation $\frac{1}{q}+\frac{1}{q^*}=1$ and $p^*$ is the same with respect to $p$.  The first integral is finite if $q^*<2$ or $q>2$ and likewise the second integral is finite if $p^*>2$ or $p<2$.
\end{proof}

\section{Existence of the Weak Solution}\label{existence}
This section contains the proof of Theorem \ref{maintheorem}.
\subsection{Moment Properties of the Approximate Sequence}\label{seqprop}

Let $\delta^\epsilon_0$ be a sequence of vortex blobs satisfying assumptions (\ref{diracIC}) and (\ref{diracstructre}).  Let $\omega_0\in L^1\cap L^\infty(\mathbb{R}^2)$ be the given initial ambient vorticity and construct $\omega^\epsilon$, $\delta^\epsilon$, $u^\epsilon$, and $h^\epsilon$ as in Subsection \ref{decomposition}.  Lemma \ref{weakapproxlemma} shows these functions are approximate solutions of the Vortex-Wave equation.  Applying Lemmas \ref{uexistence} and \ref{pathlemma} we find a subsequence $\epsilon\rightarrow 0$ along with a function $u(x,t)$ and a path $z(t)$ which satisfy all of the conclusions of the theorem with exception of the fact that $u$ and $z$ form a solution of the Vortex-Wave equation.  We only need to take the weak limits to show (\ref{epsweak}) converges to (\ref{weak}) to finish the proof. These limits are demonstrated in the follow subsection using the technical lemmas established in Section \ref{techsec}.

From here on we are considering the specific subsequence of $\epsilon$ given by Lemma \ref{uexistence}.
We record some consequences of Section \ref{momentsection}.   Recall $M_\epsilon$ and $I_\epsilon$ are defined as in (\ref{moment1}-\ref{moment2}) and $z^\epsilon(t)$ is the path under the flow $u^\epsilon$ starting at $z_0$, it is defined by (\ref{zdefn}).

Property (iv) in Lemma \ref{weakapproxlemma}, along with a well known property of the Biot-Savart kernel (see \cite[p. 315]{MR1867882}) implies the sequence $u^\epsilon$ is uniformly log-Lipschitz on $[0,T]$.  That is, $\forall s\in [0,T]$,
\begin{align}
|u^\epsilon(x,s)-u^\epsilon(y,s)|&\leq C_{\omega_0}\phi(x-y) \notag\\
C_{\omega_0}=\|\omega^\epsilon(s)\|_{L^\infty\cap L^1(\mathbb{R}^2)}&\leq \|\omega_0\|_{L^\infty\cap L^1(\mathbb{R}^2)}.\notag
\end{align}
By choosing $F=u^\epsilon$ and applying Theorem \ref{momconvthrm} (while keeping this uniform estimate in mind) we see
\begin{align}
\lim_{\epsilon\rightarrow 0}\sup_{t\in [0,T]}|M_\epsilon(t)- z^\epsilon(t)|=0.\label{Mepsconv}
\end{align}
In fact this statement follows directly from (\ref{momentconvest}).  If $I_\epsilon(0)\leq \exp(-\exp C_{\omega_0}T)$, (\ref{I2bound}) implies
\begin{align}
I_\epsilon(t)\leq eI_\epsilon(0)^{\exp(-C_{\omega_0}t)}\ \ \ \ \ \ \ \ \forall t\in[0,T].\label{uI2bound}
\end{align}

Note (\ref{Mepsconv}), (\ref{uI2bound}), and Lemma \ref{pathlemma} imply the sequence $\delta^\epsilon$ satisfies the assumptions of Lemma \ref{bigguns}.  

\subsection{Weak Limit of Approximate Sequence} \label{mainproof}

Keep in mind the remark at the end of the previous subsection: the approximate blobs $\delta^\epsilon$ satisfy the assumptions of Lemma \ref{bigguns} and $z(t)$ is a continuous path in $\mathbb{R}^2$.  Let $\Phi\in C_\sigma^\infty (\mathbb{R}^2\times [0,T]\setminus \{(z(t), t)\}_{t\in [0,T]})$.  We now show how (\ref{epsweak}) converges to (\ref{weak}) with $h(x,t)=K(x-z(t))$.

The first two limits we consider are well known for the Euler equation.  Since our velocity sequence $u^\epsilon$ corresponds to vorticity uniformly in $L^1\cap L^\infty$ the following claims are identical to those in the literature.

{\bf Claim:}
\begin{align}
\int_0^T\int_{\mathbb{R}^2}\Phi_t u^\epsilon\,dx\,dt \rightarrow \int_0^T\int_{\mathbb{R}^2}\Phi_t u\,dx\,dt\label{claim1}
\end{align}
This claim follows immediately from the fact that $\Phi_t$ has compact support and (\ref{l1convergence}).  

{\bf Claim:}
\begin{align}
\int_0^T\int_{\mathbb{R}^2}\nabla \Phi:u^\epsilon\otimes u^\epsilon\,dx\,dt \rightarrow \int_0^T\int_{\mathbb{R}^2}\nabla \Phi:u\otimes u\,dx\,dt\label{claim2}
\end{align}
This is observed by adding and subtracting the cross terms, then applying (\ref{l2convergence}).  See \cite[p. 402]{MR1867882} for more details.

Now we work on the terms which involve $h^\epsilon$ using the lemmas in the previous subsection.

{\bf Claim:}
\begin{align}
\int_0^T\int_{\mathbb{R}^2}\Phi_t h^\epsilon\,dx\,dt \rightarrow \int_0^T\int_{\mathbb{R}^2}\Phi_t h\,dx\,dt.\label{claim3}
\end{align}
The first step in evaluating this limit is to rewrite it using the the Biot-Savart kernel:
\begin{align}
&\int_0^T\int_{\mathbb{R}^2}\Phi_t(x) h^\epsilon(x)\,dx\,dt -\int_0^T\int_{\mathbb{R}^2}\Phi_t(x) h(x)\,dx\,dt\notag\\
&=\int_0^T\int_{\mathbb{R}^2}\Phi_t(x) \int_{\mathbb{R}^2}K(x-y)\delta^\epsilon(y)\,dy\,dx\,dt-\int_0^T\int_{\mathbb{R}^2}\Phi_t(x) K(x-z(t))\,dx\,dt\notag\\
&=\int_0^T\left[\int_{\mathbb{R}^2}\left(\int_{\mathbb{R}^2}\Phi_t(x) K(x-y)\,dx\right)\delta^\epsilon(y)\,dy-\int_{\mathbb{R}^2}\Phi_t(x) K(x-z(t))\,dx\right]\,dt.\notag
\end{align}
The support of $\Phi\in C_\sigma^\infty (\mathbb{R}^2\times [0,T]\setminus \{(z(t), t)\}_{t\in [0,T]})$ and (\ref{transportconservation}) justifies changing the order of integration between the last two lines.  The integrand $\Phi_t(x) K(x-z(t))$ is smooth on its support and the final integral makes sense.  Corollary \ref{equcor} applied to $\Phi_t$ shows $K\ast \Phi_t$ is uniformly continuous, uniformly on $[0,T]$.  $\Phi$ is smooth with compact support so Lemma \ref{inftylemma} implies $K\ast \Phi_t\in L^\infty([0,T]\times\mathbb{R}^2)$.  Corollary \ref{biggunscor} applied to $K\ast \Phi_t$ then shows 
\begin{align}
\int_0^T\left|\int_{\mathbb{R}^2}\left(\int_{\mathbb{R}^2}\Phi_t(x) K(x-y)\,dx\right)\delta^\epsilon(y)\,dy-\int_{\mathbb{R}^2}\Phi_t(x) K(x-z(t))\,dx\right|\,dt\rightarrow 0\notag
\end{align}
and the proof of (\ref{claim3}) is complete.

{\bf Claim:}
\begin{align}
\int_0^T\int_{\mathbb{R}^2}\nabla \Phi:h^\epsilon\otimes u^\epsilon \,dx\,dt&\rightarrow \int_0^T\int_{\mathbb{R}^2}\nabla \Phi:h\otimes u \,dx\,dt.\label{claim4}\\
\int_0^T\int_{\mathbb{R}^2}\nabla \Phi:u^\epsilon\otimes h^\epsilon \,dx\,dt&\rightarrow \int_0^T\int_{\mathbb{R}^2}\nabla \Phi:h\otimes u \,dx\,dt.\label{claim4a}
\end{align}
Since these two are nearly identical we will only prove (\ref{claim4}).  By adding and subtracting the cross term $\iint\nabla \Phi:h\otimes u^\epsilon \,dx\,dt$ we break this case into two parts.  
We will show
\begin{align}
\int_0^T\int_{\mathbb{R}^2}\nabla \Phi:h^\epsilon\otimes u^\epsilon \,dx\,dt&\rightarrow \int_0^T\int_{\mathbb{R}^2}\nabla \Phi:h\otimes u^\epsilon \,dx\,dt,\label{mixconv1}\\
\int_0^T\int\nabla \Phi:h\otimes u^\epsilon \,dx\,dt&\rightarrow \int_0^T\int\nabla \Phi:h\otimes u \,dx\,dt.\label{mixconv2}
\end{align}
Consider first (\ref{mixconv1}):
\begin{align}
\int_0^T\int_{\mathbb{R}^2}&\nabla \Phi:h^\epsilon\otimes u^\epsilon \,dx\,dt- \int_0^T\int_{\mathbb{R}^2}\nabla \Phi:h\otimes u^\epsilon \,dx\,dt\notag\\
&=\int_0^T\int_{\mathbb{R}^2}\int_{\mathbb{R}^2}\nabla \Phi(x):K(x-y)\delta^\epsilon(y)\otimes u^\epsilon(x) \,dy\,dx\,dt\notag\\
&\ \ \ \ \ \ \ \ \ -\int_0^T\int_{\mathbb{R}^2}\nabla \Phi(x):K(x-z(t))\otimes u^\epsilon(x) \,dx\,dt\notag\\
&=\int_0^T\int_{\mathbb{R}^2}\left(\int_{\mathbb{R}^2}\nabla \Phi(x):K(x-y)\otimes u^\epsilon(x)\,dx\right)\delta^\epsilon(y)\,dy\,dt \notag\\
&\ \ \ \ \ \ \ \ \ -\int_0^T \int_{\mathbb{R}^2}\nabla \Phi(x):K(x-z(t))\otimes u^\epsilon(x)\,dx\,dt.\notag
\end{align}
To justify the above sequence we rely on the differentiability of $\Phi$ and its compact support combined with (\ref{htransportconserve}) and $u^\epsilon\in L^\infty(\mathbb{R}^2)$ which follows from (\ref{transportconservation}) and Lemma \ref{inftylemma}.  Indeed,
\begin{align}
\|u^\epsilon(t)\|_{L^\infty(\mathbb{R}^2)}\leq C\|\omega^\epsilon(t)\|_{L^1\cap L^\infty(\mathbb{R}^2)}\leq C\|\omega(0)\|_{L^1\cap L^\infty(\mathbb{R}^2)}.\notag
\end{align}
Combined with the compact support of $\Phi$ we see that for any $p\geq 1$,  $\partial_i\Phi_ju_j^\epsilon\in L^\infty([0,T];L^p(\mathbb{R}^2))$ uniformly.  Hence Lemmas \ref{equlemma} and \ref{inftylemma} show
\begin{align}
f^\epsilon(y,t)=\int_{\mathbb{R}^2}\nabla \Phi(x,t):K(x-y)\otimes u^\epsilon(x,t)\,dx\notag
\end{align}
is uniformly equicontinuous, uniformly on $[0,T]$ and satisfies (\ref{eqlass4}).  Then, Lemma \ref{bigguns} then implies (\ref{mixconv1}).  

To obtain (\ref{mixconv2}) recall that $\Phi$ has compact support separate from $z(t)$ so $h$ is a bounded function in the support of $\Phi$.  The strong convergence (\ref{l1convergence}) now implies (\ref{mixconv2}).

{\bf Claim:}
\begin{align}
\int_0^T\int_{\mathbb{R}^2}\nabla \Phi:h^\epsilon\otimes h^\epsilon \,dx\,dt\rightarrow 0.\label{claim5}
\end{align}
To prove this we recall
\begin{align}
\int_0^T\int_{\mathbb{R}^2}\nabla \Phi:h\otimes h \,dx\,dt=0\label{hscalar}
\end{align}
then add and subtract a cross term so the claim is reduced to proving
\begin{align}
\int_0^T\int_{\mathbb{R}^2}\nabla \Phi:h^\epsilon\otimes h^\epsilon \,dx\,dt&\rightarrow \int_0^T\int_{\mathbb{R}^2}\nabla \Phi:h\otimes h^\epsilon \,dx\,dt,\label{hconv1}\\
\int_0^T\int\nabla \Phi:h\otimes h^\epsilon \,dx\,dt&\rightarrow \int_0^T\int\nabla \Phi:h\otimes h \,dx\,dt.\label{hconv2}
\end{align}
To check (\ref{hscalar}) one can compute directly, when $x\neq z(t)$,
\begin{align}
h(x,t)\cdot\nabla h(x,t)=K(x-z(t))\cdot\nabla K(x-z(t)) = \nabla \frac{1}{4\pi}\frac{1}{|x-z(t)|^2}.\notag
\end{align}
That is, in the support of $\Phi$, $h(x,t)\cdot\nabla h(x,t)$ is the gradient of a scalar.

The convergence (\ref{hconv2}) is argued similar to (\ref{mixconv1}) or (\ref{claim3}) once one notes $h$ is smooth and bounded in the support of $\Phi$.  The convergence (\ref{hconv1}) is more subtle and we use an argument inspired by the proof of Theorem 3.2 in \cite{MR727203}.  We proceed as before and rewrite the quantity:
\begin{align}
\int_0^T\int_{\mathbb{R}^2}&\nabla \Phi:h^\epsilon\otimes h^\epsilon \,dx\,dt -\int_0^T\int_{\mathbb{R}^2}\nabla \Phi:h\otimes h^\epsilon \,dx\,dt\notag\\
&=\int_0^T\int_{\mathbb{R}^2}\left(\int_{\mathbb{R}^2}\nabla \Phi(x):K(x-y)\otimes h^\epsilon(x)\,dx\right)  \delta^\epsilon(y) \,dy\,dt\notag\\
&\ \ \ \ \ \ \ \ -\int_0^T\int_{\mathbb{R}^2}\nabla \Phi(x):K(x-z(t))\otimes h^\epsilon(x)\,dx  \,dt.\notag
\end{align}

To apply Lemmas \ref{equlemma} and \ref{bigguns}, and hence prove the desired convergence, we need to establish $\partial_i\Phi_jh_j^\epsilon$ is in $L^\infty([0,T];L^p(\mathbb{R}^2))$ uniformly for some $p\in (2,\infty)$.  Since $\nabla \Phi$ is smooth and has compact support we will need to focus our efforts on $h^\epsilon$.  For this we decompose $h^\epsilon$ into two parts, one representing the contribution of $\delta^\epsilon$ near $z(t)$ and the other representing the remaining part.  Let \begin{align}
r_0=\frac{1}{4}\inf\{|x-z(t)|:(x,t)\in Supp(\Phi)\}.\notag                                                                                                                                                                                                                                                                                                                                                                                                                                                                                                                                                                                                                                                                                                                                                                                                                                                                                                                                                                                                                 \end{align}
By assuming $\epsilon$ is sufficiently small and relying on (\ref{Mepsconv}) we may use $|M_\epsilon(t)-z(t)|<r_0$.  Now consider
\begin{align}
h^\epsilon(x,t)&=\int_{B_{r_0}(z(t))}K(x-y)\delta^\epsilon(y)\,dy+\int_{B^c_{r_0}(z(t))}K(x-y)\delta^\epsilon(y)\,dy\notag\\
&= h^{1,\epsilon}+h^{2,\epsilon}.\notag
\end{align}
For any $(x,t)\in Supp(\Phi)$ and $y\in B_{2r_0}(M_\epsilon(t))$ we have $|K(x-y)|<\frac{1}{2\pi r_0}$ so the first integral may be bounded by
\begin{align}
|h^{1,\epsilon}(x,t)| &\leq \int_{B_{2r_0}(M_\epsilon(t))}K(x-y)\delta^\epsilon(y)\,dy\notag\\
&\leq \frac{1}{2\pi r_0}\int_{B_{2r_0}(M_\epsilon(t))}\delta^\epsilon(y)\,dy\notag\\
&\leq \frac{1}{2\pi r_0}.\label{h1epsbound}
\end{align}
This bound is uniform in $t$ and $x$ and the support of $\Phi$ is compact so we conclude $\partial_i\Phi_jh_j^{1,\epsilon}\in L^\infty([0,T];L^p(\mathbb{R}^2))$ for all $p\in [1,\infty]$.

Now we handle $h^{2,\epsilon}$.  The first step is to bound the measure of the support of $\delta^\epsilon$ outside of $B_{r_0}(z(t))$ which we label $\Sigma(t)$:
\begin{align}
\Sigma(t) = \{x\in B_{r_0}^C(z(t))| \delta^\epsilon(x,t)\neq 0\}.\notag
\end{align}
Recall (\ref{diracshape}) and denote (as in Section \ref{momentsection}) by $\Lambda^\epsilon_t$ the image of $\Lambda$ transported by $u^\epsilon$.  Then,
\begin{align}
|\Sigma(t)| &\leq \epsilon^2\int_{B^c_{\frac{r_0}{2}}(M_\epsilon(t))}\delta^\epsilon(x,t)\,dx\notag\\
&\leq \frac{\epsilon^2}{9r^2_0}\int_{B^c_{\frac{r_0}{2}}(M_\epsilon(t))}(x-M_\epsilon(t))^2\delta^\epsilon(x,t)\,dx\notag\\
&\leq \frac{\epsilon^2}{9r^2_0}I_\epsilon(t).\label{sigmaestimate}
\end{align}
where $I_\epsilon$ is the second moment controlled by (\ref{uI2bound}).  The next piece is the well known estimate of the Biot-Savart law considered as a singular integral operator.  Combining the Calderon-Zygmund inequality with the Sobolev inequality (see \cite[p. 322]{MR1867882}) shows 
\begin{align}
\|h^{2,\epsilon}(t)\|_{L^p(\mathbb{R}^2)}\leq C\|\chi_{\Sigma(t)}\delta^\epsilon(t)\|_{L^q(\mathbb{R}^2)}\label{h2epsbound}
\end{align}
where $p$ and $q$ satisfy the relation $\frac{1}{p}=\frac{1}{q}-\frac{1}{2}$. Therefore if we are able to prove $\chi_{\Sigma(t)}\delta^\epsilon(t) \in L^\infty([0,T];L^q(\mathbb{R}^2))$ for some $q\in (1,2)$ we will have shown $\partial_i \Phi_j h_i^{2,\epsilon}(t)\in L^\infty([0,T];L^p(\mathbb{R}^2))$ for some $p\in(2,\infty)$.  To show this we first claim that if $I_\epsilon(0)$ satisfies assumption (\ref{diracstructre}) then $\|\chi_{\Sigma(t)}\delta^\epsilon(t)\|_{L^q(\mathbb{R}^2)}$ is bounded uniformly for $q\in (1,1+\frac{1}{2}\exp(-C_{\omega_0}T))$ where $C_{\omega_0}$ is the same as in (\ref{uI2bound}).  Indeed, for such a $q$,
\begin{align}
\|\chi_{\Sigma(t)}\delta^\epsilon(t)\|_{L^q(\mathbb{R}^2)}&\leq \frac{1}{\epsilon^2}|\Sigma(t)|^{\frac{1}{q}}\notag\\
&\leq \frac{\epsilon^{\frac{2}{q}-2}}{(9r_0)^\frac{1}{q}}I_\epsilon(t)^\frac{1}{q}\notag\\
&\leq C_{r_0}\epsilon^{\frac{2}{q}-2+2\exp(-C_{\omega_0}T)}.\notag
\end{align}
Moving from the first to the second line above we used (\ref{sigmaestimate}).  The move from the second to the third is exactly (\ref{uI2bound}).  For $q\in (1,1+\frac{1}{2}\exp(-C_{\omega_0}T))$ the exponent of $\epsilon$ is positive and we may conclude $\|\chi_{\Sigma(t)}\delta^\epsilon(t)\|_{L^q(\mathbb{R}^2)}$ is bounded uniformly in time (in fact it tends to zero as $\epsilon\rightarrow 0$).  Then (\ref{h2epsbound}) implies $h^{2,\epsilon}\in L^\infty([0,T];L^p(\mathbb{R}^2))$ uniformly from $p\in (2,\tilde{p})$ where $\tilde{p}=(2+\exp(-C_{\omega_0}T))/(1-\frac{1}{2}\exp(-C_{\omega_0}T))$.   

When combined with (\ref{h1epsbound}) we see $h^\epsilon\in L^\infty([0,T],L^{\tilde{p}}(\mathbb{R}^2))$  for some $\tilde{p}>2$ and with the compact support of $\Phi$ we have $\partial_i\Phi_jh_j^\epsilon\in L^\infty([0,T],L^p(\mathbb{R}^2))$ for all $p\in [1,\tilde{p}]$.  Now we can apply Lemmas \ref{equlemma}, \ref{inftylemma}, and \ref{bigguns} similar to the previous terms and obtain (\ref{hconv1}).  Therefore we have proven (\ref{claim5}).  Taking all of the claims together, (\ref{claim1}-\ref{claim5}), we conclude the proof of Theorem \ref{maintheorem}.

\bibliographystyle{plain}
\bibliography{vortex_wave_local}
\end{document}